\documentclass[12pt]{amsart}
\usepackage{amsmath,amssymb,amsthm,mathtools}
\usepackage{mathrsfs,bm}
\usepackage{thmtools,thm-restate}
\usepackage{tikz-cd}
\usepackage{tikz}
\usepackage{enumitem}
\usepackage{graphicx}
\usetikzlibrary{matrix,arrows.meta,bending}
\usepackage{color}
\usepackage[numbers]{natbib}
\usepackage{booktabs}
\usepackage{longtable}
\usepackage{pgfplots}
\pgfplotsset{compat=1.14} 
\usetikzlibrary{arrows.meta,decorations.pathreplacing,positioning,calc} 

\usepackage[a4paper, left=3cm, right=3cm, top=3cm, bottom=3cm, footskip=15mm]{geometry}

\newtheorem{theorem}{Theorem}[section]

\newtheorem{lemma}[theorem]{Lemma}
\newtheorem{proposition}[theorem]{Proposition}
\newtheorem{corollary}[theorem]{Corollary}
\newtheorem{conjecture}[theorem]{Conjecture}
\theoremstyle{definition}
\newtheorem{definition}[theorem]{Definition}

\theoremstyle{remark}
\newtheorem{remark}[theorem]{Remark}

\numberwithin{equation}{section}

\makeindex

\usepackage{fancyhdr}
\usepackage{calc} 

\AtBeginDocument{%
  \setlength{\textwidth}{\dimexpr\paperwidth - 6cm\relax}%
  \setlength{\oddsidemargin}{\dimexpr 3cm - 1in\relax}%
  \setlength{\evensidemargin}{\dimexpr 3cm - 1in\relax}%
  \setlength{\marginparwidth}{2.5cm}%
}

\begin{document}

\title{The Theory of ramification}

\author{Theophilus Agama}
\address{Department of Mathematics, African Institute for Mathematical science, Ghana
}
\email{theophilus@aims.edu.gh/emperordagama@yahoo.com}



\date{\today}



\footnote{
\par
}%

\begin{abstract}
In this paper, we introduce and develop the concept of \emph{ramification} in a given modulus. We study some properties in relation to this concept and it's connection to some important problems in mathematics, particularly the Goldbach conjecture.
\end{abstract}

\maketitle

\section{Introduction}

The elementary - looking definition that we adopt-an integer $n$ \emph{ramifies} in the modulus $m$ when the residue of $n$ modulo $m$ pairs with a residue of $n$ modulo a strictly smaller modulus so that the two residues add to $m$ - provides a simple, concrete language to record interactions of congruences at two different scales. Intuitively, if the image of an object in a mirror of size $m$ can be complemented by the image produced in a smaller mirror $r<m$ so that the concatenation fills the larger mirror, then the object behaves as a ramifier. This combinatorial viewpoint is deliberately elementary, but it also connects directly to deep additive questions and to the study of residues and representations of integers.\\

Framing the binary Goldbach conjecture in the language of ramification is the principal motivation of this work: the conjecture can be restated as the assertion that every even $m\geq 6$ admits a \emph{strong} ramifier, meaning that there exists $n$ whose residues modulo $m$ and some $r<m$ are both prime and sum to $m$. While this reformulation may not make the problem easier, it recasts Goldbach as a problem about compatible residue-pairs and highlights the specific kinds of modular and multiplicative information one would need to settle the conjecture.\\

Two classical approaches have dominated progress on representations of integers as sums of primes: the Hardy--Littlewood circle method and sieve techniques. The circle method, introduced by G. H. Hardy and J. E. Littlewood, gives asymptotic formulas for representations of integers as sums of structured sets (notably primes) and underlies the celebrated Vinogradov-style results on sums of primes. The sieve approach, typified by Chen's theorem, produces partial but powerful results by constructing almost-prime representations. These pillars inform the background intuition behind our counting and density arguments: our elementary upper and lower bounds (see §3) show precisely where analytic exponential-sum inputs or sieve lower bounds would be inserted to strengthen conclusions about strong ramifiers.\\

What is new here is the systematic introduction of ramifiers and the related vocabulary (index of ramification, circle of ramification, ramification character). Rather than attempting a direct analytic assault on Goldbach, we develop a modular-combinatorial framework that:

\begin{itemize}
  \item isolates structural constraints on how residues at different moduli can pair to form a fixed modulus $m$,
  \bigskip
  
  \item yields elementary existence results via descending and congruence arguments, and
  \bigskip
  
  \item produces quantitative upper and lower bounds for the counting function 
  $$
  \#\{n\leq x:\mathcal R(m)=n\}
  $$ 
  in terms of $x$ and $m$.
\end{itemize}
\bigskip

These elementary estimates (Theorems \ref{sup} and \ref{inf} in the paper) both serve as evidence that the ramification framework is non-vacuous and as a clear bookkeeping device showing which analytic or sieve improvements would be necessary to obtain results on \emph{strong} ramification.\\

Informally, the main counting idea is simple: an integer $n$ has residues $a_1 \pmod m$ and $a_2 \pmod r$; when $a_1+a_2=m$, we call $n$ a ramifier for $m$. Controlling the number of such $n$ therefore reduces to controlling congruence systems that force $n$ to simultaneously lie in certain arithmetic progressions. This combinatorial viewpoint makes elementary obstructions apparent (for example Proposition \ref{ram} shows that ramifiers cannot be congruent to $0 \pmod m$ and suggests natural places where multiplicative-order information or nontrivial exponential-sum bounds can sharpen naive density estimates (see Theorem \ref{quadratic} and its use in improving upper bounds).\\

The ramification character $\kappa_m(n)$ introduced in §6 is an indicator function analogous to objects studied routinely in additive number theory and sieve theory; partial-sum estimates for $\kappa_m$ therefore echo classical counting problems treated in standard references. While we do not attempt a full analytic attack on strong ramification here, the language developed makes it straightforward to plug in stronger inputs (for instance bounds for character sums, exponential sums, or sieve lower bounds) to obtain improved asymptotics for the ramifier counts.

\subsection{Organization of the paper}

For the convenience of the reader, we give a short roadmap of the sections that follow:
\begin{itemize}
  \item \textbf{\S2.} Elementary properties and basic existence results for ramifiers.
  \bigskip
  
  \item \textbf{\S3.} Upper and lower counting bounds for 
  $$
  \#\{n\leq x:\mathcal R(m)=n\}
  $$ 
  and related density remarks.
  \bigskip
  
  \item \textbf{\S4.} The notion of \emph{strong} ramification and the reformulation of the binary Goldbach conjecture in this language; discussion of classical partial results that illuminate the gap between current knowledge and the strong-ramifier statement.
  \bigskip
  
  \item \textbf{\S5.} The index of ramification, structural constraints and magnification-type phenomena (Theorem \ref{magnification}).
  \bigskip
  
  \item \textbf{\S6.} The circle of ramification, the ramification character $\kappa_m$, and multiplicative-type properties with partial-sum estimates.
  \bigskip
  
  \item \textbf{\S7.} Final remarks and directions for further research, including where analytic or sieve-theoretic inputs would strengthen the conclusions about strong ramifiers.
\end{itemize}

\subsection*{Acknowledgments and background references}

The reader who wishes to place this work in the classical literature could consult the foundational Hardy--Littlewood papers on the circle method, the Vinogradov treatment of sums of primes, and standard modern texts on sieve methods and multiplicative number theory \cite{tenenbaum2015introduction, hildebrand2005introduction, hardy1923some, inogradov1947method, chen2002representation}. In particular, we use the language of indicator functions and partial-sum estimates that appears in standard references on analytic number theory and sieve methods.
\bigskip

\section{The concept of ramification}

\begin{definition}\label{concept}
Let $n\geq 2$ be an integer and $n\equiv a_1\pmod m$. The integer $n$ is said to ramify in $\pmod m$ if there exist some $r<m$ with $n\equiv a_2 \pmod r$ so that $a_1+a_2=m$. We say that the modulus $m$  admits a ramifier and we denote the ramifier by $\mathcal{R}(m)=n$.
\end{definition}

\begin{remark}
Definition \ref{concept} has a practical implication. The concept affirms the notion that, given the image of an object on a mirror of a certain size, If we can find a mirror of a relatively smaller size that produce an image of the same body so that the concatenation of the two covers the size of the  larger mirror, then the body must indeed be a ramifier. Next, we examine some properties of the ramifier in a given modulus.
\end{remark}

\section{Properties of the ramifier}

In this section, we study some properties of the ramifier in a fixed modulus. We also count the number of ramifiers in all modulus. We first give a proof that indicates that there must exist a ramifier in any given modulus. The method of proof employs in an ingenious way an infinite descending argument whose consequence is not suitable for that particular regime.

\begin{proposition}
There exist a ramifier in a fixed modulus. In particular, for any $m\geq 2$, there exists a ramifier in $\pmod t$ for a fixed $1<t\leq m$.
\end{proposition}

\begin{proof}
Suppose on the contrary that for all $m\geq 2$, then the modulus do not admit a ramifier for all  $1<t\leq m$. Then it follows by definition \ref{concept} that there exist some sequence of positive integers $2=s_1<s_2<\ldots <s_k=m$ such that for all $m$ with $n\equiv a_1\pmod m$ 
\begin{align}
m\neq a_1+r_i\nonumber
\end{align}
where $n\pmod {s_i}=r_i$ for $i=1,\ldots k-1$. Again there exist some $1<r_j\leq r_{k-1}$ such that $a_1+r_j<m$ if and only if $r_j<m-a_1<m$. Now choose $t_k=m-a_1<m$, then by assumption it follows that for $n\equiv a_2\pmod {t_k}$ so that there exist a sequence of positive integers $t_k>v_{k-1}>v_{k-2}>\cdots v_1>1$ such that  $a_2+u_i\neq t_k$ for all $i=1,2 \ldots k-1$, where $n\pmod {v_i}=u_i$. It follows that there exist some $1<u_j\leq u_{k-1}$ so that  $a_2+u_j<t_k$ if and only if $u_j<t_k-a_2<t_k$. By choosing $t_k-a_2=t_{k-1}<t_k<m$ and using the fact that each $1<t\leq m$ admits no ramifier, we obtain by induction an infinite descending sequence of positive integers
\begin{align}
m>t_{k}>t_{k-1}>t_{k-2}>\cdots >t_{k-i}>\cdots. \nonumber
\end{align}
This proves the proposition.
\end{proof}

\begin{remark}
The next result highlights a sufficient condition for any positive integer to ramify in a given modulus.
\end{remark}

\begin{proposition}\label{ram}
Let $m\geq 2$. If $\mathcal{R}(m)=n$ then $\mathcal{R}(m)\not \equiv 0\pmod m$. 
\end{proposition}

\begin{proof}
Let $m\geq 2$ and let $\mathcal{R}(m)=n$. Suppose on the contrary that $\mathcal{R}(m)\equiv 0\pmod m$, then it follows that for the sequence $m=r_{k}>r_{k-1}>\ldots >r_{1}>1$, where $\mathcal{R}(m)\pmod {r_i}=s_i$ with $i=1,2,\ldots k-1$, it must be that $s_i+0<m$. This contradicts the fact that $m$ admits a ramifier. This completes the proof of the proposition.
\end{proof}
\bigskip

Proposition \ref{ram}, although simple, is somewhat revealing. It allows us to control at the very least the number of ramifiers for a finite set of integers in a given modulus. That is to say, for any set of the form 
$$
\{n\leq x:\mathcal{R}(m)=n\}
$$
then 
\begin{align}
\# \{n\leq x:\mathcal{R}(m)=n\}&=\sum \limits_{\substack{n\leq x\\\mathcal{R}(m)=n}}1\nonumber \\&\leq x-\bigg\lfloor \frac{x}{m}\bigg\rfloor \nonumber \\&=\bigg(1-\frac{1}{m}\bigg)x+O(1).\nonumber
\end{align}
It follows from this upper bound that the distribution of ramifiers in any finite set of the integers depends greatly on the modulus of ramification. It is clear that the smaller the modulus, the less chance there is to find a ramifier in the set. Conversely, the larger the modulus the high chance there is in picking a ramifier in the set in any random selection. This upper bound, although very weak, could serve as a benchmark. Applying Proposition \ref{ram} indicates that we can do better than this if we knew other subtle properties of the ramifiers in any finite set of the integers. The sequel will be focused on studying such properties.

\begin{theorem}\label{quadratic}
Let $p$ be a prime and let $(a,p)=1$. If $a$ is a quadratic residue modulo $p$,  then the set 
$$
\mathcal{M}:=\{a, a^2, \ldots, a^{p-1}\}
$$ 
contains at least two non-ramifiers modulo $p$.
\end{theorem}

\begin{proof}
Let $p$ be a prime and $(a,p)=1$. It follows that $a^{p-1}\equiv 1 \pmod p$. It follows immediately that $\mathcal{R}(p) \neq a^{p-1}$. If we assume that $a$ is a quadratic residue modulo $p$, then it follows that 
\begin{align}
a^{\frac{p-1}{2}}\equiv 1\pmod p\nonumber
\end{align}
and it follows that $\mathcal{R}(p)\neq a^{\frac{p-1}{2}}$, thereby ending the proof.
\end{proof}

\begin{remark}
In light of Theorem \ref{quadratic}, we can improve on the upper bound in the preceding discussion concerning the scale of ramifiers in a given modulus.
\end{remark}

\begin{theorem}\label{sup}
Let $m$ be fixed and let  
$$
\mathcal{I}:=\{n\leq x:\mathcal{R}(m)=n\}.
$$
We have
\begin{align}
\# \mathcal{I}\leq \bigg(1-\frac{1}{m}\bigg)x-\frac{\log x}{\log m}+O(1).\nonumber
\end{align}
\end{theorem}

\begin{proof}
In the preceding discussion, the number of ramifiers that led to the upper bound are integers $n\leq x$ satisfying $n\equiv 0\pmod m$. Let 
$$
\mathcal{I}:=\{n\leq x:\mathcal{R}(m)=n\}
$$ 
be the set of ramifiers in modulo $m$. By Theorem \ref{quadratic}, we find that the upper bound be can slightly improved to \begin{align}
\# \mathcal{I} &\leq \bigg(1-\frac{1}{m}\bigg)x-\sum \limits_{\substack{a\leq x\\ a^{k}\leq x\\a^{k}\equiv 1\pmod m\\(a,m)=1}}1+O(1)\nonumber \\&=\bigg(1-\frac{1}{m}\bigg)x-\sum \limits_{\substack{a\leq x\\(a,m)=1}}\sum \limits_{\substack{a^{k}\equiv 1 \pmod m\\1\leq k \leq \lfloor \frac{\log x}{\log a}\rfloor}}1+O(1)\nonumber
\end{align}
and the result follows by taking $a=m+1$ in the sum.
\end{proof}

\begin{remark}
In connection with the binary Goldbach conjecture, we launch a very strict form of the notion of Ramifiers. The Goldbach conjecture can be formulated in this language. It comes in the following sequel.
\end{remark}

\begin{definition}
Let $n\geq 2$ be an integer and $n\equiv p_1\pmod m$. The integer $n$ is said to ramify \emph{strongly} in $\pmod m$ if there exist some $r<m$ such that $n\equiv p_2 \pmod r$, such that $p_1+p_2=m$ where $p_1,p_2$ are all prime. 	In other words, we say that the modulus $m$ admits a strong ramifier.
\end{definition}

\begin{conjecture}[Goldbach]
Every even number $n\geq 6$ admits a strong ramifier in $\pmod n$.
\end{conjecture}

\begin{theorem}\label{inf}
There are infinitely many ramifiers in $\pmod m$ for some fixed $m$.
\end{theorem}

\begin{proof}
It suffices to obtain a lower bound for the quantity 
$$
\#\{n\leq x:\mathcal{R}(m)=n\}.
$$ 
It follows that 
\begin{align}
\# \{n\leq x:\mathcal{R}(m)=n\}&=\sum \limits_{\substack{n\leq x\\\mathcal{R}(m)=n}}1\nonumber \\&=\sum \limits_{\substack{n\leq x\\a_0+b_0=m\\n\equiv a_0\pmod m\\n\equiv b_0\pmod {r_0}\\r_0<m}}1\nonumber \\&=\sum \limits_{\substack{n\leq x\\a_0+b_0=m\\mr_0|(n-a_0)(n-b_0)\\r_0<m}}1 \nonumber \\&=\sum \limits_{\substack{n\leq x\\a_0+b_0=m\\r_0<m}}\sum \limits_{mr_0|(n-a_0)(n-b_0)}1\nonumber \\&=\sum \limits_{\substack{a_0+b_0=m\\r_0<m}}\left \lfloor \frac{(x-a_0)(x-b_0)}{mr_0}\right \rfloor \nonumber \\&=\sum \limits_{\substack{a_0+b_0=m\\r_0<m}}\frac{x^2-x(a_0+b_0)+a_0b_0}{mr_0}+O_m(1)\nonumber \\& \geq \frac{x^2-xm}{m^2}+O_m(1)\nonumber
\end{align}
and the result follows immediately from this estimate.
\end{proof}
\bigskip

The above lower bound for the number of ramifiers in a fixed modulus is somewhat instructive. It puts a threshold on the size of the modulus that cannot admit a ramifier from a finite set of the integers $n\leq x$. Indeed, for this lower bound to fail, the inequality must be satisfied 
\begin{align}
\frac{x^2-xm}{m^2}+O_m(1)>x\bigg(1-\frac{1}{m}\bigg)-\frac{\log x}{\log m}+O(1).\nonumber
\end{align}
Using the main term, it follows that 
\begin{align}
m<\frac{x}{\sqrt{x-\log x}}.\nonumber
\end{align}
Thus, the modulus for which the lower bound majorizes the upper bound for the number of ramifiers in a finite set gives the largest scale of a modulus that do not admit a ramifier. It follows that the size of any modulus that admits a ramifier in any finite set of the integers $n\leq x$ must  satisfy the inequality 
\begin{align}
m\geq \left \lfloor \frac{x}{\sqrt{x-\log x}}\right \rfloor +1.\nonumber
\end{align}

\begin{remark}
Next, we prove a result that suggests that there are some integers  $n\leq x$ that ramifies in more than one modulus $m<x$. We find the following elementary estimate useful: 
\end{remark}

\begin{lemma}\label{lem 1}
We have
\begin{align}
\sum \limits_{n=1}^{\infty}\frac{1}{n^2}=\frac{\pi^2}{6}.\nonumber
\end{align}
\end{lemma}

\begin{proof}
For a proof see, for example, \cite{hildebrand2005introduction}.
\end{proof}

\begin{lemma}\label{lem 2}
We have
\begin{align}
\sum \limits_{n\leq x}\frac{1}{n}=\log x+\gamma +O\bigg(\frac{1}{x}\bigg).\nonumber
\end{align}
\end{lemma}

\begin{proof}
For a proof see, for example, \cite{tenenbaum2015introduction}.
\end{proof}

\begin{theorem}
We have 
\begin{align}
\sum \limits_{m\leq x}\sum \limits_{\substack{n\leq x\\\mathcal{R}(m)=n}}1\geq \frac{x\log x}{2}+O(x).\nonumber
\end{align}
\end{theorem}

\begin{proof}
We observe that by an application of Lemma \ref{lem 1}, Lemma \ref{lem 2} and Theorem \ref{inf}
\begin{align}
\sum \limits_{\frac{x}{\sqrt{x-\log x}}<m\leq x}\sum \limits_{\substack{n\leq x\\\mathcal{R}(m)=n}}1&\geq \sum \limits_{\frac{x}{\sqrt{x-\log x}}<m\leq x}\frac{x^2-xm}{m^2}+O(x) \nonumber \\& =x^2\sum \limits_{\frac{x}{\sqrt{x-\log x}}<m\leq x}\frac{1}{m^2}-x\sum \limits_{\frac{x}{\sqrt{x-\log x}}<m\leq x}\frac{1}{m}+O(x)\nonumber \\&=x^2\bigg(\sum \limits_{m>\frac{x}{\sqrt{x-\log x}}}\frac{1}{m^2}-\sum \limits_{m>x}\frac{1}{m^2}\bigg)-x\sum \limits_{\frac{x}{\sqrt{x-\log x}}<m\leq x}\frac{1}{m}+O(x)\nonumber \\&=O(x)+O(1)+x\log (\sqrt{x-\log x})+O(\sqrt{x})\nonumber \\&=x\log (\sqrt{x-\log x})+O(x)\nonumber \\&=\frac{x\log x}{2}+O(x).\nonumber
\end{align}
\end{proof}

\begin{corollary}
There exist at least one integer $n\leq x$ that ramifies in at least two modulus $m\leq x$.
\end{corollary}

\begin{proof}
The result follows from the pigeon-hole principle.
\end{proof}
\bigskip

\section{The index of ramification}
In this section, we launch the notion of the \emph{index} of ramification. We expose some relationship between ramifiers and their corresponding indices.

\begin{definition}
Let $n\geq 2$ be a positive integers that ramifies in modulo $m\geq 2$. By the \emph{index} of ramification in modulo $m$, denoted $\mathrm{ind}_{m}(n)$, we mean the value $r_j<m$ so that for $n\equiv a_i\pmod m$, then $n\equiv s_j\pmod {r_j}$ such that $a_i+s_j=m$.
\end{definition}

\begin{theorem}\label{magnification}
Let $n \equiv a_i \pmod m$ and suppose that $(n-m,a_i)=1$. If $\mathcal{R}(m)=n$, then $\mathrm{ind}_{m}(n)\equiv 0\pmod {a_i}$ or $(\mathrm{ind}_{m}(n), a_i)=1$.
\end{theorem}

\begin{proof}
Let $n \equiv a_i \pmod m$ with $(n-m,a_i)=1$ and suppose for the sake of contradiction that $(\mathrm{ind}_m(n), a_i)=d$ with $1<d<a_i$. It follows that $\bigg(\frac{\mathrm{ind}_m(n)}{d}, \frac{a_i}{d}\bigg)=1$. Since $\mathcal{R}(m)=n$, it follows that there exist some $r_k<m$ such that for $n\equiv s_k\pmod {r_k}$, then $a_i+s_k=m$. It follows that $d|(m-s_k)$. Since $d|\mathrm{ind}_{m}(n)$, it follows that $d|(n-s_k)$. Thus, it follows that $d|(n-m)$. This contradicts the assumption $(n-m,a_i)=1$, since $d|a_i$ and $1<d<a_i$.
\end{proof}

\begin{remark}
Theorem \ref{magnification}, roughly speaking, suggests that the image of a body in a mirror of somewhat large size could be magnified to cover the size of a certain smaller mirror.
\end{remark}

\section{The circle of ramification}

In this section, we launch the notion of the circle of ramification in a given modulus. We formally launch the following language:

\begin{definition}
Let 
$$
\mathcal{I}:=\{n\leq x:\mathcal{R}(m)=n\}
$$ 
be any set of ramifiers. By the \emph{circle of ramification} relative to $\mathcal{I}$ with center $m$ and radius $r$, we mean  $|\mathcal{R}(m)-m|\leq r$, where $r=\mathrm{max}\{|\mathcal{R}(m)-m|\}$.
\end{definition}

\begin{remark}
The next result suggests that for any finite set of the integers, we can get control on the radius of the circle of ramification. In other words, there appears to be lack of degree of freedom in constructing circles of ramification, given any finite set of integers.
\end{remark}

\begin{proposition}\label{circ}
Let 
$$
\mathcal{I}:=\{n\leq x:\mathcal{R}(m)=n\}
$$
be any set of ramifiers, then 
\begin{align}
\mathrm{max}\{|\mathcal{R}(m)-m|\}& \leq \frac{x(\sqrt{x-\log x}-1)}{\sqrt{x-\log x}}.\nonumber
\end{align}
\end{proposition}

\begin{proof}
The result follows by applying Theorem \ref{inf} and the previous discussion on the least scale of modulus that admits a ramifier.
\end{proof}

\begin{remark}
Proposition \ref{circ} suggests that the ramifiers in any finite set must not be too far way from the centre of ramification, in the sense that they must be closer to the centre than expected with distance $\leq x^{1-\epsilon}$ for some $\epsilon>0$.
\end{remark}

\section{Ramification character}

It is important to note that in a given modulus not every integer is a ramifier. In other words, there are some numbers that ramify and some that do not ramify in a given modulus. A sequel to this paper will be focused on developing a criterion to decide which number is a ramifier for any given modulus. In this section, however, we launch the ramification character and establish some elementary properties in this regard.

\begin{definition}(Ramification character)
Let $n$ be any positive integer. Then we set 
\begin{align}
\kappa_{m}(n):=
\begin{cases}1 \quad \text{if}\quad \mathcal{R}(m)=n\\0 \quad \text{otherwise.}
\end{cases}\nonumber
\end{align}
\end{definition}

\begin{remark}
For the remaining part of the paper, we will study some interesting properties of the ramification character in a given modulus.
\end{remark}

\begin{proposition}
Let $m$ be a fixed positive integer. The following properties of the ramification character holds: 
\begin{enumerate}

\item [(i)] $\kappa_{m}(n+2m)=\kappa_{m}(n)$.
\bigskip

\item [(ii)] $\kappa_{m}(n+m!)=\kappa_{m}(n)$.
\bigskip

\item [(iii)] $\kappa_{m}(1)=0$.
\bigskip

\item [(iv)] $\kappa_{m}(n)=0$ for $n\equiv 0,1\pmod m$.
\bigskip

\item [(v)] $\kappa_{m}(nm!)=\kappa_{m}(n)\kappa_{m}(m!)$.
\end{enumerate}
\end{proposition}
\bigskip

\begin{proof}
We prove only $(ii)$, $(iii)$, $(iv)$ and $(v)$. For $(ii)$, since $n+m!\equiv n\pmod {r_i}$ for any sequence $r_0<r_1<\ldots r_{k-1}<r_{k}=m$ the result follows immediately according as $n$ is a ramifier or a non-ramifier. Clearly $(iii)$ and $(iv)$ follow from Proposition \ref{ram} and Proposition \ref{quadratic}. Finally $(iv)$ is also easy to establish.
\end{proof}
\bigskip

A natural quest is to seek various upper and lower bounds for the partial sums of the ramification character in a fixed modulus. That is, we seek estimates for sums of the form 

\begin{align}
\sum \limits_{n\leq x}\kappa_{m}(n).\nonumber
\end{align}
It is easy to check trivial upper and lower bounds for this sum have been established in Theorem \ref{sup} and Theorem \ref{inf}, by observing that 
\begin{align}
\sum \limits_{\substack{n\leq x\\\mathcal{R}(m)=n}}1&=\sum \limits_{n\leq x}\kappa_{m}(n).\nonumber
\end{align} 
We obtain the following weaker estimate for the partial sums of the ramification character as follows:

\begin{theorem}
Let $m$ be a fixed positive integer, then the inequality \begin{align}
\frac{x^2-xm}{m^2}+O_m(1)& \leq \sum \limits_{n\leq x}\kappa_m(n) \leq \bigg(1-\frac{1}{m}\bigg)x-\frac{\log x}{\log m}+O(1)\nonumber
\end{align}
hold for all 
\begin{align}
m\geq \left \lfloor \frac{x}{\sqrt{x-\log x}}\right \rfloor +1.\nonumber
\end{align}
\end{theorem}

\begin{proof}
The result follows by combining Theorem \ref{inf} and Theorem \ref{sup}.
\end{proof}

\section{Final remarks}
In this paper, we have introduced the concept of the ramifiers. We have established some properties and some consequences of this theory. The binary Goldbach conjecture, which is an important open problem, can be framed in this language as:

\begin{conjecture}[Goldbach]
Every even number $n\geq 6$ admits a strong ramifier in $\pmod n$.
\end{conjecture}

\bibliographystyle{amsplain}

\end{document}